\documentclass[12pt,a4paper]{article}

\usepackage{amssymb, amsmath, amsthm}

\usepackage[cm]{fullpage}
\usepackage[english]{babel}
\usepackage[pdftex]{graphicx}

\usepackage{hyperref}

\newtheorem{thm}{Theorem}

\title{On a discrete composition of the fractional integral and Caputo derivative\footnote{This is an accepted version of the manuscript published in \textit{Communications in Nonlinear Science and Numerical Simulations} \textbf{108} (2022), 106234 with DOI: \url{https://doi.org/10.1016/j.cnsns.2021.106234}}}
\author{\L ukasz P\l ociniczak\thanks{Faculty of Pure and Applied Mathematics, Wroc{\l}aw University of Science and Technology, Wyb. Wyspia{\'n}skiego 27, 50-370 Wroc{\l}aw, Poland}$\;^,$\footnote{Email: lukasz.plociniczak@pwr.edu.pl}}
\date{}

\begin{document}
\maketitle

\begin{abstract}
	We prove a discrete analogue for the composition of the fractional integral and Caputo derivative. This result is relevant in numerical analysis of fractional PDEs when one discretizes the Caputo derivative with the so-called L1 scheme. The proof is based on asymptotic evaluation of the discrete sums with the use of the Euler-Maclaurin summation formula. \\
	
	\noindent\textbf{Keywords}: fractional integral, Caputo derivative, Euler-Maclaurin formula\\
	
	\noindent\textbf{AMS Classification}: 26A33, 34A08, 65R20
\end{abstract}
	
\section{Introduction}
Let $I^\alpha$ be the fractional integral operator of order $\alpha\in (0,1)$, i.e. for any locally integrable function $y: (0,T) \mapsto \mathbb{R}$ we define
\begin{equation}
	I^\alpha y(t) = \frac{1}{\Gamma(\alpha)} \int_0^t (t-s)^{\alpha-1} y(s) ds. 
\end{equation}
Further, the Caputo derivative is defined on smooth functions by
\begin{equation}
	D^\alpha y(t) = I^{1-\alpha} y'(t) = \frac{1}{\Gamma(1-\alpha)} \int_0^t (t-s)^{-\alpha} y'(s)ds, \quad 0<\alpha<1.
\end{equation}
An elementary result (see for ex. \cite{Li19}) states that the composition $I^\alpha D^\alpha$ is given by
\begin{equation}
\label{eqn:CompositionC}
	I^\alpha D^\alpha y(t) = I^\alpha I^{1-\alpha} y'(t) = I^1 y'(t) = y(t) - y(0),
\end{equation}
which is a straightforward generalization of the fundamental theorem of calculus. In what follows we are interested in finding an analogue of the above relation when we allow for the time variable to take only a discrete number of possible values. This is especially relevant in numerical analysis where one constructs various schemes for approximately solving differential equations with fractional derivatives. For example, finite difference or finite element methods lead to a nonlocal recurrence relations that may be inverted by the use of the fractional integral or the fractional version of the discrete Gr\"onwall's lemma \cite{Lia19, Fer12, Plo17, Plo19}. Some results related to the same family as ours can be found in \cite{Kop20} where authors consider the stability of the L1 scheme on graded meshes. In particular, they give estimates for the case with power function on the right-hand side of (\ref{eqn:CompositionC}) in the discrete context. For more information concerning fractional calculus and numerical methods we refer the reader to \cite{Die02, Li19, Li19a}.

\section{Main result}
In numerical analysis, some very common discretizations of the above operators are constructed by simple quadratures. Fix a time step $h$ and define the mesh
\begin{equation}
	t_n = n h, 
\end{equation}
where $h>0$ is the time step. For brevity we denote $y_n := y(t_n)$. Some very useful discretizations of $I^\alpha$ and $D^\alpha$ can be constructed by a simple rectangle quadrature applied to defining integrals. In particular, we have
\begin{equation}
	I^\alpha y_n = J^{\alpha} y_n + Q_n, \quad D^\alpha = \delta^\alpha y_n + R_n,
\end{equation}
where the discretizations $J^\alpha$ and $\delta^\alpha$ are defined by
\begin{equation}
\label{eqn:Discretizations}
	J^\alpha y_n = \frac{h^\alpha}{\Gamma(1+\alpha)} \sum_{i=0}^{n-1} b_{n-i}(\alpha) y_{i+1}, \quad \delta^\alpha y_n = \frac{h^{-\alpha}}{\Gamma(2-\alpha)} \sum_{i=0}^{n-1} b_{n-i}(1-\alpha) (y_{i+1}-y_i),
\end{equation}
with weights
\begin{equation}
\label{eqn:Weights}
	b_j(\beta) = j^\beta - (j-1)^\beta. 
\end{equation}
Furthermore, the remainders satisfy
\begin{equation}
	Q_n \sim -\frac{h}{\Gamma(1+\alpha)} \left(\frac{1}{2} t^\alpha + \frac{\zeta(-\alpha)}{n^\alpha} + \frac{\alpha}{12n}\right)t^\alpha  y'(\tau), \quad R_n \sim -\frac{h^{2-\alpha}}{\Gamma(2-\alpha)}\zeta(\alpha-1)y''(\tau), \quad n\rightarrow\infty, \, n h \rightarrow t, 
\end{equation}
where $\zeta$ is Riemann-Zeta function and $\tau\in (0,t)$ is some number. The bounds above are sharp (see \cite{Plo21a}). In the literature this discretization of the fractional derivative is called the L1 scheme (see \cite{Old74}). 

We are interested in discrete version of the composition formula (\ref{eqn:CompositionC}), that is we expect that $J^\alpha \delta^\alpha y_n = y_n - y_0 + r_n$, with some remainder $r_n$. In order to prove this result and find the form of $r_n$ we need to recall the Euler-Maclaurin formula written in the form that we need (for a proof see for ex. \cite{Lam01}).

\begin{thm}[Euler-Maclaurin]
For $f\in C([0,m])$ we have
\begin{equation}
\label{eqn:EM}
	\sum_{k=1}^m f(k) = \int_1^m f(x) dx + \frac{1}{2}\left(f(1)+f(m)\right) + \int_1^m f'(x) P_1(x) dx,
\end{equation}
where $P_1(x) = B_1(x-\lfloor x\rfloor)$ is the periodized Bernoulli polynomial $B_1(x) = x - \frac{1}{2}$. 
\end{thm}

We can now proceed to our main result. Notice that in the below the remainder is of order $h^{\min(\alpha,1-\alpha)}$, that is, the exponent is always not larger than $1/2$. This interesting fact comes from the singularity of kernels. When composed, the fractional integral and Caputo derivative produce a kernel that has two types of singularities at each endpoint of the integration interval. 
\begin{thm}
Let $\alpha\in(0,1)$ and $y\in C^1(0,T)$. For any fixed $t\in(0,T)$ with $n\rightarrow\infty$ and $h = t/n$ we have
\begin{equation}
\label{eqn:Composition}
	J^\alpha \delta^\alpha y_n = y_n - y_0 + r_n
\end{equation}
where the remainder $r_n$ satisfies
\begin{equation}
\label{eqn:Remainder}
	|r_n|\leq C h^\beta \int_0^t (t-s)^{-\beta} |y'(s)| ds, \quad \beta:= \min(\alpha,1-\alpha).
\end{equation} 
and the constant $C$ depends only on $\alpha$ and $y$. 
\end{thm}
\begin{proof}
Fix $n = 1, ..., N$ and start with writing the composition as
\begin{equation}
\begin{split}
	J^\alpha \delta^\alpha y_n
	&= \frac{h^\alpha}{\Gamma(1+\alpha)} \sum_{i=0}^{n-1} b_{n-i}(\alpha) \delta^\alpha y_{i+1}\\
	&=\frac{1}{\Gamma(1+\alpha)\Gamma(2-\alpha)} \sum_{i=0}^{n-1} b_{n-i}(\alpha) \sum_{j=0}^{i} b_{i-j+1}(1-\alpha) (y_{j+1}-y_j).
\end{split}
\end{equation}
Now, interchanging the order of summation we can write
\begin{equation}
\begin{split}
\label{eqn:CompositionS}
	J^\alpha \delta^\alpha y_n &=\frac{1}{\Gamma(1+\alpha)\Gamma(2-\alpha)}\sum_{j=0}^{n-1} \left(\sum_{i = j}^{n-1}  b_{n-i}(\alpha)b_{i-j+1}(1-\alpha) \right) (y_{j+1}-y_j) \\
	&= \frac{1}{\Gamma(1+\alpha)\Gamma(2-\alpha)}\sum_{j=0}^{n-1} \left(\sum_{k = 1}^{n-j}  b_{n-j-k+1}(\alpha)b_{k}(1-\alpha) \right) (y_{j+1}-y_j) 
\end{split}
\end{equation}
where in the second equality we have introduced a new summation variable $k = i-j+1$ for the sum in the parenthesis. Put $m = n - j$ and denote the sum in the parenthesis above
\begin{equation}
	S_m = \sum_{k = 1}^{m}  b_{m-k+1}(\alpha)b_{k}(1-\alpha).
\end{equation}
We would like to find the approximation of $S_m$ and to this end we use Euler-Maclaurin formula (\ref{eqn:EM}) by using (\ref{eqn:Weights}) and writing
\begin{equation}
\label{eqn:EMSum0}
\begin{split}
	S_m 
	&= \sum_{k = 1}^{m} \left((m-k+1)^\alpha-(m-k)^\alpha\right)\left(k^{1-\alpha}-(k-1)^{1-\alpha}\right) \\
	&= m \sum_{k = 1}^{m} \left(\left(1-\frac{k}{m}+\frac{1}{m}\right)^\alpha-\left(1-\frac{k}{m}\right)^\alpha\right)\left(\left(\frac{k}{m}\right)^{1-\alpha}-\left(\frac{k}{m}-\frac{1}{m}\right)^{1-\alpha}\right) \\
	&= m \sum_{k = 1}^{m} f(k),
\end{split}
\end{equation}
where we defined the function $f$. Now, using (\ref{eqn:EM}) we can write
\begin{equation}
\label{eqn:EMSum}
	S_m = m\int_1^m f(x) dx + \frac{m}{2}\left(f(1) + f(m)\right) + m\int_1^m f'(x) P_1(x) dx.
\end{equation}
We will estimate the above three components when $m$ is large. We quickly can see that the easiest part is
\begin{equation}
\label{eqn:Average}
\begin{split}
	\frac{m}{2}\left(f(1) + f(m)\right) 
	&= \frac{m}{2} \left( \left(\frac{1}{m}\right)^{1-\alpha}\left(1-\left(1-\frac{1}{m}\right)^\alpha\right) + \left(\frac{1}{m}\right)^\alpha \left(1-\left(1-\frac{1}{m}\right)^{1-\alpha}\right) \right)\\
	&= O\left(\frac{1}{m^{1-\alpha}}+\frac{1}{m^\alpha} \right), \quad m\rightarrow\infty,
\end{split}
\end{equation}
where the asymptotic behaviour follows from the Taylor expansion. Further, we turn to the analysis of the first term in (\ref{eqn:EMSum}) for which we change the integration variable $y = m x$
\begin{equation}
	m\int_1^m f(x) dx = m^2\int_\frac{1}{m}^1 \left(\left(1-y+\frac{1}{m}\right)^\alpha-\left(1-y\right)^\alpha\right)\left(y^{1-\alpha}-\left(y-\frac{1}{m}\right)^{1-\alpha}\right) dy.
\end{equation}
Note that the above is invariant under the transformation $\alpha \mapsto 1-\alpha$ what can be seen by a substitution $y = 1-x+1/m$. By inspection we can see that the integrand converges to $\alpha(1-\alpha) (1-y)^{\alpha-1}y^{-\alpha}$ as $m\rightarrow\infty$ (each expression in parenthesis converges to its derivative), hence, by the Lebesgue Monotone Convergence Theorem and the definition of Euler beta function we conclude that
\begin{equation}
	m\int_1^m f(x) dx \rightarrow \alpha(1-\alpha) \int_0^1(1-y)^{\alpha-1}y^{-\alpha} dy = \Gamma(1+\alpha)\Gamma(2-\alpha) \quad \text{as} \quad m\rightarrow\infty.
\end{equation}
This, together with (\ref{eqn:CompositionS}) proves that in this limit, the discrete composition verifies (\ref{eqn:CompositionC}). However, we would like to investigate the rate of this convergence to have a more useful formula. Due to singularity of the integrand, we cannot expand it into Taylor series for large $m$ and then integrate since such expansion would diverge. Since the behaviour of the integrand is different for each terminal: $t=0$, $t=1$ we split it into two terms
\begin{equation}
\label{eqn:MainPart}
	m\int_1^m f(x) dx = m^2 \left(\int_\frac{1}{m}^\frac{1}{2} + \int_\frac{1}{2}^1\right) = K_0+K_1.
\end{equation}
Because $1/m > 0$, the integrand in $K_0$ does not have any singularities we can safely expand for $1/m \rightarrow 0$
\begin{equation}
	K_0 = \alpha(1-\alpha)\int_\frac{1}{m}^\frac{1}{2} (1-y)^{\alpha-1}y^{-\alpha} dy + \frac{\alpha(1-\alpha)}{2m} \int_\frac{1}{m}^\frac{1}{2}(1-y)^{\alpha-2}y^{-\alpha-1} (\alpha-y) dy + O\left(\frac{1}{m^2}\right).
\end{equation} 
And the first integral above is
\begin{equation}
\begin{split}
	\int_\frac{1}{m}^\frac{1}{2} (1-y)^{\alpha-1}y^{-\alpha} dy &= \int_0^\frac{1}{2} (1-y)^{\alpha-1}y^{-\alpha} dy - \int_0^\frac{1}{m} (1-y)^{\alpha-1}y^{-\alpha} dy \\
	&= \int_0^\frac{1}{2} (1-y)^{\alpha-1}y^{-\alpha} dy - \frac{1}{1-\alpha} \frac{1}{m^{1-\alpha}} + O\left(\frac{1}{m^{2-\alpha}}\right),
\end{split}
\end{equation}
since the term $(1-y)^{\alpha-1} = O(1)$ as $m\rightarrow\infty$. By the same argument, the second integral in $K_0$ can be expanded as follows
\begin{equation}
	\int_\frac{1}{m}^\frac{1}{2}(1-y)^{\alpha-2}y^{-\alpha-1} (\alpha-y) dy = m^\alpha + O(1), \quad m\rightarrow\infty. 
\end{equation}
Therefore,
\begin{equation}
	K_0 = \alpha(1-\alpha) \int_0^\frac{1}{2} (1-y)^{\alpha-1}y^{-\alpha} dy + O\left(\frac{1}{m^{1-\alpha}}\right), \quad m\rightarrow\infty. 
\end{equation}
A similar analysis cannot be conducted for $K_1$ since we would arrive at a divergence. A roundabout can be constructed by observing that 
\begin{equation}
	\left(1-y+\frac{1}{m}\right)^{\alpha}-(1-y)^\alpha = \alpha \int_0^\frac{1}{m} (1-y+z)^{\alpha-1} dz.
\end{equation}
Then, by Tonelli's theorem and noting that $m(y^{1-\alpha} - (y-1/m)^{1-\alpha}) = (1-\alpha) y^{-\alpha} + O(1/m)$ we have
\begin{equation}
\begin{split}
	K_1 
	&= \alpha m \int_0^\frac{1}{m} \left(\int_\frac{1}{2}^1 (1+z-y)^{\alpha-1}  y^{-\alpha} dy\right) \left(1+O\left(\frac{1}{m}\right)\right)dz \\
	&= \alpha m \int_0^\frac{1}{m} \left(\int_\frac{1}{2(1+z)}^\frac{1}{1+z} (1-u)^{\alpha-1} u^{-\alpha} du\right) \left(1+O\left(\frac{1}{m}\right)\right)dz,
\end{split}
\end{equation}
where we have put $y = (1+z)u$ what removes the singularity from the integrand. The integral in parenthesis can now be expanded for $z\rightarrow 0^+$ (since $z\in (0,1/m))$ yielding the leading order
\begin{equation}
	K_1 = \alpha \int_\frac{1}{2}^1 (1-u)^{\alpha-1} u^{-\alpha} du + O\left(\frac{1}{m^\alpha}\right), \quad m\rightarrow\infty.
\end{equation}
Finally, we can go back to (\ref{eqn:MainPart}) to obtain
\begin{equation}
\label{eqn:fIntAsym}
	m \int_1^m f(x) dx = \Gamma(1+\alpha) \Gamma(2-\alpha) + O\left(\frac{1}{m^\alpha} + \frac{1}{m^{1-\alpha}}\right), \quad m\rightarrow\infty. 
\end{equation}
The next step is to proceed with the remainder in (\ref{eqn:EMSum}). Its analysis is similar to the above and we sketch only the most important details. By calculating derivatives we see that the remainder has almost the same form as before
\begin{equation}
\begin{split}
	m\int_1^m & f'(x) P_1(x) dx \\
	&= -\alpha m \int_\frac{1}{m}^1 \left(\left(1-y+\frac{1}{m}\right)^{\alpha-1}-\left(1-y\right)^{\alpha-1}\right)\left(y^{1-\alpha}-\left(y-\frac{1}{m}\right)^{1-\alpha}\right)P_1(my) dy \\
	& - (1-\alpha) m \int_\frac{1}{m}^1 \left(\left(1-y+\frac{1}{m}\right)^{\alpha}-\left(1-y\right)^{\alpha}\right)\left(y^{-\alpha}-\left(y-\frac{1}{m}\right)^{-\alpha}\right)P_1(my) dy.
\end{split}
\end{equation}
By counting powers and utilizing the fact that $P_1$ is bounded we can specify the correct convergence order of the above. For example, the first integral has two singularities for large $m$: $y = 0$ and $y=1$. In the former case we can expand in the Taylor series for $m\rightarrow\infty$ which will consume two powers of $m$ leaving $y^{-\alpha}$ singularity. After integration we obtain a term proportional to $m$ with an exponent $1-2+1-\alpha = -\alpha$. On the other hand, to deal with the singularity at $y=1$ we use the trick with Tonelli's theorem to obtain a $O(m^{-\alpha})$ term. The second integral above can be analysed in the same way with the difference that the singularity at $y = 1$ yields a $O(m^{-1})$ term (because the integrand after expansion in $m\rightarrow\infty$ is integrable there), while the one at $y = 0$ produces $O(m^{\alpha-1})$. We see that the remainder introduces terms of the same order, hence putting everything together, recalling that $m=n-j$, and returning to (\ref{eqn:CompositionS}) brings us to
\begin{equation}
	J^\alpha \delta^\alpha y_n = y_n-y_0 + \sum_{j=0}^{n-1} \frac{c_{n-j}}{(n-j)^\beta} (y_{j+1}-y_j), \quad \beta = \min(\alpha,1-\alpha),
\end{equation}
with $c_{n-j}$ bounded by, say, $C$. Hence, for $n\rightarrow\infty$ with $nh \rightarrow t \in (0,T)$ by the definition of Riemann integral we have
\begin{equation}
\begin{split}
	\left|\sum_{j=0}^{n-1} \frac{c_{n-j}}{(n-j)^\beta} (y_{j+1}-y_j)\right| &\leq C h \sum_{j=0}^{n-1} \frac{|y'(\xi_{j})|}{(n-j)^\beta} \\
	&= C \frac{n h}{n^{1+\beta}} \sum_{j=0}^{n-1} \left(1-\frac{j}{n}\right)^{-\beta}|y'(\xi_{j})| \sim C \frac{t}{n^\beta} \int_0^1 (1-x)^{-\beta} |y'(x t)| dx \\
	&\sim C h^\beta \int_0^t (t-s)^{-\beta} |y'(s)| ds.
\end{split}
\end{equation}
with $\xi_j$ being a intermediate point and the new integration variable $s = xt$. The proof is complete.
\end{proof}
We close the paper with several numerical verifications of the above theorem. First, it is interesting to see how the asymptotic relation for the Euler-Maclaurin's integral (\ref{eqn:fIntAsym}) behaves. In Fig. \ref{fig:fIntAsym} a loglog plot of
\begin{equation}
\label{eqn:fError}
	\left|m\int_1^m f(x) dx - \Gamma(1+\alpha)\Gamma(2-\alpha)\right|,
\end{equation}
is depicted for increasing $m$ with $f$ defined in (\ref{eqn:EMSum0}). Recall that the integrand is invariant under the transformation $\alpha \mapsto 1-\alpha$ and thus we consider only one exemplary case of $\alpha=0.75$. As can be seen, the integral approaches its limit with the correct rate. Numerical simulations with other values of $\alpha$ give very similar results.

\begin{figure}
	\centering
	\includegraphics[scale = 1]{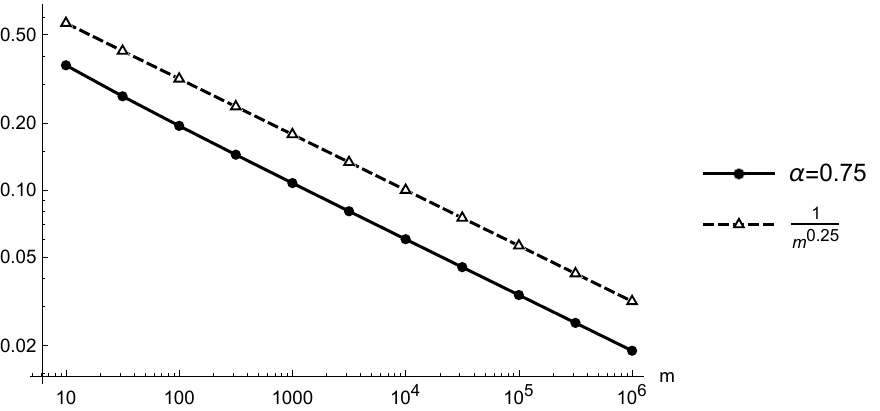}
	\caption{The loglog plot of (\ref{eqn:fError}) for increasing $m$ and $\alpha=0.75$ (solid, circles) and the reference line (dashed, triangles).}
	\label{fig:fIntAsym}
\end{figure}

We can also verify the main relation of this paper, that is the discrete composition (\ref{eqn:Composition}). As an example we choose three test functions: a polynomial $t^3$, $\sin t$, and a non-smooth function $|t-1/2|$. We compute the residue
\begin{equation}
\label{eqn:Residue}
	\rho = \left| J^\alpha \delta^\alpha y_n - y_n+y_0\right|,
\end{equation}
with the final time $t=1$, $\alpha=0.5$, and a decreasing sequence of steps $h = 1/n$. Results of calculations are presented in Fig. \ref{fig:Residue}. Immediately we can see that now, the convergence to zero is not monotone. Rather, for the majority of chosen steps $h$ the results of computations cluster along or parallel to the reference line $h^{0.5}$. This confirms the predicted order of the remainder $r_n$ in (\ref{eqn:Remainder}). As can also be seen, the results are not sensitive to the chosen test function even if it does not have a continuous derivative. Interestingly, the residue for $h=10^{-i}$ for $i=1,2,3,...$ is smaller than in the other cases. In each case, however, our estimate of the remainder (\ref{eqn:Remainder}) is confirmed. 

\begin{figure}
	\centering
	\includegraphics[scale = 1]{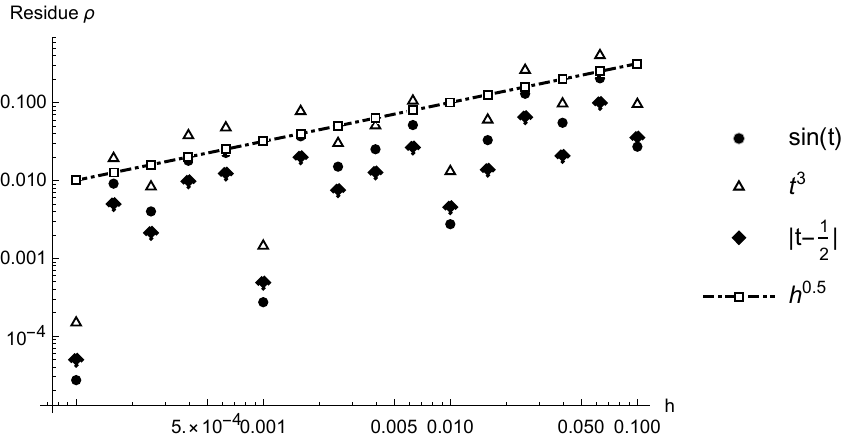}
	\caption{The loglog plot of (\ref{eqn:Residue}) as a function of $h=1/n$ with a fixed value of $\alpha=0.5$ for different test functions. }
	\label{fig:Residue}
\end{figure}

\section*{Acknowledgement}
Ł.P. has been supported by the National Science Centre, Poland (NCN) under the grant Sonata Bis with a number NCN 2020/38/E/ST1/00153.

\bibliography{biblio}
\bibliographystyle{plain}
	
\end{document}